\documentclass[a4paper, oneside,11pt]{amsart}

\usepackage[english]{babel}
\usepackage{mathrsfs,amssymb}
\usepackage{mathtools}
\usepackage[colorlinks, citecolor = blue]{hyperref}
\usepackage{tcolorbox}
\usepackage[shortlabels]{enumitem}
\usepackage{esint}
\usepackage[ruled, lined, procnumbered, linesnumbered]{algorithm2e}
\usepackage{scalerel}
\usepackage{tikz}
\usepackage{graphicx}

\usetikzlibrary{arrows.meta}
\definecolor{qone}{RGB}{220,76,70}
\definecolor{qtwo}{RGB}{45,123,210}
\definecolor{qthree}{RGB}{246,196,68}
\definecolor{atomA}{RGB}{150,96,116}
\definecolor{atomB}{RGB}{238,136,74}
\definecolor{atomC}{RGB}{120,160,139}
\definecolor{atomD}{RGB}{246,196,68}
\definecolor{unusedcolor}{RGB}{185,185,185}
\definecolor{flowcolor}{RGB}{34,113,181}

\usepackage{todonotes}

\newtheorem{theorem}{Theorem}[section]
\newtheorem{lemma}[theorem]{Lemma}
\newtheorem{prop}[theorem]{Proposition}

\theoremstyle{definition}
\newtheorem{definition}[theorem]{Definition}

\theoremstyle{remark}
\newtheorem{remark}[theorem]{Remark}
\newtheorem{example}[theorem]{Example}

\numberwithin{equation}{section}

\newcommand{\Z}{\ensuremath{\mathbb{Z}}}

\newcommand{\R}{\ensuremath{\mathbb{R}}}

\newcommand{\mc}{\mathcal}
\newcommand{\ms}{\mathscr}

\DeclarePairedDelimiter\abs{\lvert}{\rvert}

\DeclarePairedDelimiter\cbrace\{\}
\DeclarePairedDelimiter\ha()

\newcommand{\hab}[1]{\bigl(#1\bigr)}
\newcommand{\cbraceb}[1]{\bigl\{#1\bigr\}}

\newcommand{\bcup}[1]{{\textstyle\bigcup\limits_{#1}}}
\newcommand{\bcap}[1]{{\textstyle\bigcap\limits_{#1}}}

\newcommand{\has}[1]{\Bigl(#1\Bigr)}
\newcommand{\cbraces}[1]{\Bigl\{#1\Bigr\}}

\newcommand{\dd}{\hspace{2pt}\mathrm{d}}

\newcommand{\ee}{\mathrm{e}}

\DeclareMathOperator{\ind}{\mathbf{1}}

\allowdisplaybreaks

\newcommand{\m}[1]{\ensuremath{\mu\left(#1\right)}}

\begin{document}
\title[Optimization algorithms for Carleson and sparse collections]{Optimization algorithms for Carleson and sparse collections of sets}

\author[E.A. Honig]{Eline A. Honig}
\address[E.A. Honig]{Delft Institute of Applied Mathematics\\
	Delft University of Technology \\ P.O. Box 5031\\ 2600 GA Delft\\The
	Netherlands}
\email{e.a.honig@tudelft.nl}
\author[E. Lorist]{Emiel Lorist}
\address[E. Lorist]{Delft Institute of Applied Mathematics\\
	Delft University of Technology \\ P.O. Box 5031\\ 2600 GA Delft\\The
	Netherlands}
\email{e.lorist@tudelft.nl}

\thanks{The authors would like to thank Dion Gijswijt for his insightful comments, in particular his suggestion to use submodular function minimization. Moreover, the second author would like to thank Timo H\"anninen and Guillermo Rey for fruitful discussions on sparse and Carleson collections and Floris Roodenburg for careful proofreading a draft of this paper.
The second author was partially financed by the Dutch Research Council (NWO) on the project ``The sparse revolution for stochastic partial differential equations'' with project number \href{https://doi.org/10.61686/ZGRMR99948}{VI.Veni.242.057}.}
	
\begin{abstract}
Carleson and sparse collections of sets play a central role in dyadic harmonic analysis. We employ methods from optimization theory to study such collections.

 First, we present a strongly polynomial algorithm to compute the Carleson constant of a collection of sets, improving on the recent approximation algorithm of Rey \cite{Re22}. Our algorithm is based on submodular function minimization.

Second, we provide an algorithm showing that any Carleson collection is sparse, achieving optimal dependence of the respective constants and thus providing a constructive proof of a result of H\"anninen \cite{Haen18}.
Our key insight is a reformulation of the duality between the Carleson condition and sparseness in terms of the duality between the maximum flow and the minimum cut in a weighted directed graph.
\end{abstract}
	


\maketitle

\section{Introduction}
Let $(\Omega,\Sigma,\mu)$ be a measure space and $\mc{F}$ a collection of sets in $\Sigma$, where we use the word \emph{collection} for an unordered sequence of sets that allows for repeated elements. We call $\mc{F}$ \emph{$\eta$-sparse} for $\eta\in (0,1]$ if for every $Q \in \mathcal{F}$ there exists a measurable subset $E_Q \subseteq Q$ such that $\mu(E_Q) \geq \eta \mu(Q)$ and the sets $\{E_Q\}_{Q\in\mathcal{F}}$ are pairwise disjoint. The number $\eta$ provides a scale-invariant way to quantify the overlap of the sets in $\mc{F}$. In particular, a $1$-sparse collection is pairwise disjoint (up to $\mu$-null sets).

Closely related to $\eta$-sparseness is the $\Lambda$-Carleson condition. The collection $\mc{F}$ is said to satisfy the \emph{$\Lambda$-Carleson condition} for $\Lambda\geq 1$ if for every subcollection $\mathcal{A} \subseteq \mathcal{F}$ we have
    \begin{equation*}
        \sum_{Q\in\mathcal{A}}\mu\ha{Q} \leq \Lambda \mu\hab{\bcup{Q\in\mathcal{A}}Q}.
    \end{equation*}
Every $\eta$-sparse collection $\mc{F}$ satisfies the $\eta^{-1}$-Carleson condition, as for any subcollection $\mathcal{A} \subseteq \mathcal{F}$ we have
\begin{equation*}
    \sum_{Q\in\mathcal{A}}\m{Q} \leq \frac{1}{\eta} \sum_{Q\in\mathcal{A}}\m{E_Q} \leq \frac{1}{\eta} \mu\hab{ \bcup{Q\in\mathcal{A}}Q},
\end{equation*}
where the last inequality holds because the sets $\{E_Q\}_{Q\in\mathcal{F}}$ are pairwise disjoint.

In the converse direction, it is known in many cases that if $\mc{F}$ satisfies the $\Lambda$-Carleson condition, then $\mc{F}$ is $\frac1\Lambda$-sparse. For $\mc{F}$ consisting of dyadic cubes in $\R^d$ and $\mu$ a non-atomic, locally finite Borel measure, this was proven by Verbitsky in \cite[Corollary 2]{Ve96}, using a result of Dor \cite{Do75}. Verbitsky's proof was extended to general collections of Borel measurable subsets of $\R^d$ by H\"anninen \cite{Haen18}, see also \cite[Theorem 3.3]{Ba19}.

The equivalence between sparseness and the Carleson condition is fundamental in applications in harmonic analysis: the Carleson condition is often straightforward to verify, while sparseness is incredibly useful for obtaining sharp estimates. We refer to the survey \cite{Pe19c} and the references therein for an overview of the use of sparse collections in modern-day harmonic analysis, which has been dubbed ``The sparse revolution''.

\bigskip

The remarkably elegant Dor--H\"anninen--Verbitsky proof is based on duality and the Hahn--Banach separation theorem and is thus non-constructive, even if the collection $\mc{F}$ is finite. In stark contrast, if $\mc{F}$ is a finite collection of dyadic cubes in $\R^d$, one can start by choosing any subset $E_Q \subseteq Q$ with measure $\frac1\Lambda\mu(Q)$ for all minimal cubes in $\mc{F}$. Afterwards, one can go up ``layer by layer'', choosing subsets $E_Q$ of measure $\frac{1}{\Lambda}\mu(Q)$ in $Q\setminus \bigcup_{R\in \mc{F}:R\subsetneq Q}E_R$. This will always be possible since
$$
\mu\has{\bigcup_{R\in \mc{F}:R\subsetneq Q}E_R} = \tfrac{1}{\Lambda} \sum_{R\in \mc{F}:R\subseteq Q}\mu(R)-\tfrac{1}{\Lambda} \mu(Q)\leq    (1-\tfrac{1}{\Lambda})\mu(Q),
$$
where the
 $\Lambda$-Carleson condition was used in the second step. These $E_Q$'s will be disjoint by the nestedness property of dyadic cubes. Hence, one can construct the $E_Q$'s using a greedy algorithm. This argument was extended to infinite collections of dyadic cubes in $\R^d$ with the Lebesgue measure by Lerner and Nazarov \cite[Lemma 6.3]{LN15} and subsequently with any non-atomic, locally finite Borel measure $\mu$ by Cascante and Ortega \cite[Theorem 4.3]{CO17}.

The constructive proof for dyadic cubes does not generalize to more general collections $\mc{F}$, not even to axis-parallel dyadic rectangles in two dimensions. This is due to the fact that the Carleson condition is non-local, as shown by Carleson in \cite{ca74}.
The first constructive algorithm beyond dyadic cubes was recently obtained by Rey \cite{Re22}, who  proposed a greedy \emph{approximation} algorithm to construct the sets $\cbrace{E_Q}_{Q \in \mc{F}}$ for a finite $\Lambda$-Carleson collection in $\R^d$. This algorithm constructs sets $\cbrace{E_Q}_{Q \in \mc{F}}$ such that $$\mu(E_Q) \geq \frac{c}{\Lambda \log(\ee+\Lambda)} \mu(Q)$$ for some absolute constant $0<c<1$. Under a geometric condition on $\mc{F}$, the logarithmic loss can be removed, recovering the optimal, non-constructive result from \cite{Haen18} up to an absolute constant. The algorithm of Rey can also be used to approximate the optimal Carleson condition constant of a finite collection $\mc{F}$ in polynomial time.

\bigskip

Given a measure space $(\Omega,\Sigma,\mu)$ and a finite collection $\mc{F}$ of sets in $\Sigma$, the goal of this article is twofold:
\begin{enumerate}[(i)]
  \item \label{it:goal1}  We provide a strongly polynomial algorithm to compute the minimum $\Lambda \geq 1$ such that $\mc{F}$ satisfies the $\Lambda$-Carleson condition. This will be done in Section \ref{sec:carleson}, using the fact that the Carleson condition is submodular.
  \item Given that $\mc{F}$ satisfies the $\Lambda$-Carleson condition, we  construct sets $\cbrace{E_Q}_{Q \in \mc{F}}$ such that $\mu(E_Q) \geq \frac{1}{\Lambda} \mu(Q)$, i.e. show that $\mc{F}$ is $\frac1\Lambda$-sparse. To do so, we will reformulate the duality between the Carleson condition and sparseness in terms of the duality between the maximum flow and the minimum cut in a weighted directed graph, see Section~\ref{sec:sparse}.
\end{enumerate}
Compared to \cite{Re22}, we remove the geometric assumption on $\mc{F}$ and achieve $c=1$, i.e., we calculate instead of approximate the Carleson constant and our sets $\cbrace{E_Q}_{Q \in \mc{F}}$ are optimal. This provides a constructive proof of the result from \cite{Haen18} for finite collections $\mc{F}$. We will make some remarks on countable collections $\mc{F}$ in Remark \ref{rem:infinite}.  Moreover, we extend the theory from $\R^d$ with a locally finite Borel measure to a general measure space $(\Omega,\Sigma,\mu)$, which we fix throughout the paper.

\begin{remark}
One can replace $\cbrace{\mu(Q)}_{Q \in \mc{F}}$ by an arbitrary sequence of nonnegative real numbers  $\cbrace{\lambda_Q}_{Q \in \mc{F}}$ in the definition of the Carleson condition and  sparseness, as done in, e.g., \cite{CO17, Haen18,Ve96}. Our proofs carry over verbatim to this generality.
\end{remark}

\section{The optimal Carleson constant}\label{sec:carleson}
Let $\mc{F}$ be a finite, nonempty collection of sets in the $\sigma$-algebra $\Sigma$ with $0<\mu(Q)<\infty$ for all $Q \in \mc{F}$ to avoid trivialities. Suppose we want to calculate the minimum $\Lambda \geq 1$ such that $\mc{F}$ satisfies the $\Lambda$-Carleson condition, i.e., calculate
\begin{equation*}
  \Lambda_{\mc{F}}:= \sup\cbraces{\frac{\sum_{Q\in\mathcal{A}}\mu\ha{Q}}{\mu\hab{\textstyle\bigcup_{Q\in\mathcal{A}}Q}}
:\mc{A} \subseteq \mc{F} \text{ nonempty}}.
\end{equation*}
A naive way is to use brute force and evaluate the right-hand side for all $2^{\abs{\mc{F}}}$ subcollections of $\mc{F}$, which is exponential in $\abs{\mc{F}}$  and hence impractical for large collections. We would prefer to have an algorithm that is polynomial in $\abs{\mc{F}}$ and furthermore does not depend on the size of e.g. $\mu(Q)$ for $Q \in \mc{F}$.

This leads us to the question how to represent the input of such an algorithm. Indeed, to calculate $\Lambda_{\mc{F}}$ we need to have access to the measure of $\bigcup_{Q\in\mathcal{A}}Q$ for $\mc{A} \subseteq \mc{F}$. If we explicitly list all these values, only reading the input would already take exponential time. To avoid this, we assume to have an \emph{evaluation oracle} available:
\begin{enumerate}[($\mathrm{O}_1$)]
  \item : Given a subcollection $\mc{A}\subseteq \mc{F}$, return $\mu\hab{\bigcup_{Q\in\mathcal{A}}Q}$. \label{O1}
\end{enumerate}
We aim for an algorithm with a polynomial number of calls to \ref{O1} and a polynomial amount of other computational steps. The time needed for one call to the oracle \ref{O1} is denoted by $\mathrm{EO}_1$.

Our algorithm will be based on \emph{submodular function minimization}. Let us quickly recall the definition of submodularity. We denote the powerset of $\mc{F}$ by $\mc{P}(\mc{F})$.
\begin{definition}
A set function $f \colon \mc{P}(\mc{F})\to \R$ is called \emph{submodular} if for every $\mc{A},\mc{B} \subseteq \mc{F}$ we have
\begin{align*}
  f(\mc{A})+f(\mc{B}) \geq f(\mc{A}\cup \mc{B} ) +f(\mc{A}\cap \mc{B} ).
\end{align*}
\end{definition}
There is an extensive literature on the minimization of submodular functions, i.e., finding $\mc{A} \subseteq \mc{F}$ such that $f(\mc{A}) = \min_{\mc{B} \subseteq \mc{F}} f(\mc{B})$. In particular, there exist strongly polynomial algorithms to find such minimizers, see the surveys \cite{Iw08,mc05}, the more recent \cite{JLSZ24,Ji22} and the references therein.

The submodular function we will use  is
 $$
  f_\Lambda(\mc{A}) := \Lambda \cdot \mu\hab{\bcup{Q\in\mathcal{A}}Q} - \sum_{Q \in \mc{A}} \mu(Q), \qquad \Lambda \geq 1.
  $$

\begin{lemma}\label{lemm:submodular}
 The function $f_\Lambda\colon \mc{P}(\mc{F}) \to \R$ is submodular for $\Lambda \geq 1$.
\end{lemma}

\begin{proof}
  Let $\mc{A},\mc{B} \subseteq \mc{F}$. Then we have
  \begin{align*}
    \mu\hab{\bcup{Q\in\mathcal{A}}Q} + \mu\hab{\bcup{Q\in\mathcal{B}}Q} &= \mu\has{\hab{\bcup{Q\in \mathcal{A}}Q}\cup \hab{\bcup{Q\in\mathcal{B}}Q}}+ \mu\has{\hab{\bcup{Q\in\mathcal{A}}Q} \cap \hab{\bcup{Q\in\mathcal{B}}Q}}
     \\ &\geq \mu\hab{\bcup{Q\in\mathcal{A}\cup \mathcal{B}}Q}+ \mu\has{{\bcup{Q\in\mathcal{A} \cap \mathcal{B}}Q}}.
  \end{align*}
  Moreover, we trivially have
  $$
  \sum_{Q \in \mc{A}} \mu(Q)+\sum_{Q \in \mc{B}} \mu(Q) = \sum_{Q \in \mc{A}\cup \mc{B}} \mu(Q)+\sum_{Q \in \mc{A}\cap \mc{B}} \mu(Q),
  $$
  which proves the lemma.
\end{proof}

\begin{algorithm} \label{alg:Carleson}
\DontPrintSemicolon
    \caption{Compute $\Lambda_{\mc{F}}$.}
    \SetKwInOut{Input}{Input}\SetKwInOut{Output}{Output}
    \Input{A finite, nonempty collection $\mc{F}$ of sets in $\Sigma$.}
    \Output{The constant $\Lambda$ and the collection $\mc{A}$.}
    {
     Set $\mc{A}_0 = \mc{F}$.\;
    Set $r = -1$.\;
    Set $j=0$.\;
    }
    \While{$r<0$ \label{algorithmline: for Q in F findEQ}
    }{
    Update $j \leftarrow j+1$.\;
    Set $\Lambda_{j} = \frac{\sum_{Q\in\mathcal{A}_{j-1}}\mu\ha{Q}}{\mu\hab{\bigcup_{Q\in\mathcal{A}_{j-1}}Q}}.$ \label{line:defLambda}\;
       Let $\mc{A}_{j}$ be a non-empty minimizer for  $\min_{\mc{B}\subseteq \mc{A}_{j-1}} f_{\Lambda_{j}}(\mc{B})$.\label{line:minimzer}\;
       Update $r \leftarrow f_{\Lambda_{j}}(\mc{A}_{j})$.
    }
    Set $\Lambda = \Lambda_j$.\;
    Set $\mc{A} = \mc{A}_j$.\;
\end{algorithm}

The algorithm to compute $\Lambda_{\mc{F}}$ we propose is given in Algorithm \ref{alg:Carleson}. The idea is to iteratively find the collection $\mc{A} \subseteq \mc{F}$ violating the $\Lambda$-Carleson condition the most and then update $\Lambda$ accordingly. We force the violating collections to be shrinking, ensuring that Algorithm \ref{alg:Carleson} does at most $\abs{\mc{F}}$ iterations of the while loop.

\begin{theorem}\label{thm:Lambda}
 Let $\mc{F}$ be a finite, nonempty collection in $\Sigma$ with $0<\mu(Q)<\infty$ for all $Q \in \mc{F}$. The output $(\Lambda,\mc{A})$ of Algorithm \ref{alg:Carleson} satisfies
\begin{equation}\label{eq:toprove}
  \Lambda_{\mc{F}} = \Lambda = \frac{\sum_{Q\in\mathcal{A}}\mu\ha{Q}}{\mu\hab{\textstyle\bigcup_{Q\in\mathcal{A}}Q}},
\end{equation}
 Algorithm \ref{alg:Carleson} is strongly polynomial, with a runtime of $$O(n^4 \log n \cdot \mathrm{EO}_1+ n^5 \log n),$$ where $n = \abs{\mc{F}}.$
\end{theorem}

\begin{proof}
By line \ref{line:defLambda} of  Algorithm \ref{alg:Carleson}, we have  $f_{\Lambda_{j}}(\mc{A}_{j-1})=0$ for all $j\geq 1$.  Therefore, if $\mc{A}_{j}=\mc{A}_{j-1}$ for some $j\geq 1$, we have $r=0$ and the while loop terminates. Thus, since $\mc{A}_{j} \subseteq \mc{A}_{j-1}$ by construction, the while loop is iterated at most $n$ times. Denote the number of iterations by $m$. Moreover, $f_{\Lambda_{j}}(\mc{A}_{j-1})=0$ also ensures that we can always take $\mc{A}_j$ nonempty.

Let $(\Lambda,\mc{A})$ be the output of  Algorithm \ref{alg:Carleson}. The second equality in \eqref{eq:toprove} follows directly from the fact  that $\mc{A}$ is a nonempty collection satisfying $f_{\Lambda}(\mc{A}) = r =0$  when the algorithm terminates. The first equality in \eqref{eq:toprove} is equivalent to
\begin{equation}\label{eq:toshow}
  \min_{\mc{B}\subseteq \mc{F}} f_{\Lambda}(\mc{B}) = 0.
\end{equation}
We claim that $\mc{A}_j$ is a minimizer for $\min_{\mc{B}\subseteq \mc{F}} f_{\Lambda_{j}}(\mc{B})$ for all $1\leq j \leq m$, which for $j=m$ implies \eqref{eq:toshow}, since $f_{\Lambda}(\mc{A}) =f_{\Lambda_m}(\mc{A}_m) = r =0$ when the algorithm terminates.

We prove the claim inductively. The case $j=1$ is clear since $\mc{A}_0=\mc{F}$. Suppose that the claim holds for some $j\geq 1$.  By line \ref{line:minimzer} of  Algorithm \ref{alg:Carleson} and  $f_{\Lambda_{j}}(\mc{A}_{j-1})=0$ we  have
$$
\Lambda_j \cdot \mu\hab{\textstyle\bigcup_{Q\in\mathcal{A}_{j}}Q} - \sum_{Q\in\mathcal{A}_{j}}\mu\ha{Q} =f_{\Lambda_j}(\mc{A}_{j})=  \min_{\mc{B}\subseteq \mc{A}_{j-1}} f_{\Lambda_{j}}(\mc{B})\leq0.
$$
It follows that
\begin{equation}\label{eq:Lambdaincreasing}
\Lambda_{j+1} = \frac{\sum_{Q\in\mathcal{A}_{j}}\mu\ha{Q}}{\mu\hab{\textstyle\bigcup_{Q\in\mathcal{A}_{j}}Q}} \geq \Lambda_j, \qquad j\geq 1.
\end{equation}
Let $\mc{B}'$ be a minimizer for $\min_{\mc{B}\subseteq \mc{F}} f_{\Lambda_{j+1}}(\mc{B})$. Then we have by the submodularity from Lemma \ref{lemm:submodular}, \eqref{eq:Lambdaincreasing} and the induction hypothesis
\begin{align*}
  f_{\Lambda_{j+1}}(\mc{A}_{j} \cap \mc{B}') - f_{\Lambda_{j+1}}(\mc{B}') &\leq f_{\Lambda_{j+1}}(\mc{A}_{j}) - f_{\Lambda_{j+1}}(\mc{A}_{j} \cup \mc{B}')\\  &=f_{\Lambda_{j}}(\mc{A}_{j}) -f_{\Lambda_{j}}(\mc{A}_{j} \cup \mc{B}') \\&\hspace{1cm}-(\Lambda_{j+1}-\Lambda_{j})\has{\mu\hab{\bcup{Q\in\mathcal{A}_{j}\cup\mc{B}'}Q}-\mu \hab{\bcup{Q\in\mathcal{A}_{j}}Q}}\\
   &\leq f_{\Lambda_{j}}(\mc{A}_{j}) -f_{\Lambda_{j}}(\mc{A}_{j} \cup \mc{B}') \leq 0.
\end{align*}
We conclude
\begin{align*}
  \min_{\mc{B}\subseteq \mc{F}} f_{\Lambda_{j+1}}(\mc{B}) = f_{\Lambda_{j+1}}(\mc{B}') \geq f_{\Lambda_{j+1}}(\mc{A}_{j} \cap \mc{B}')  \geq \min_{\mc{B}\subseteq \mc{A}_j} f_{\Lambda_{j+1}}(\mc{B}),
\end{align*}
which means that, since $\mc{A}_{j+1}$ is a minimizer for $\min_{\mc{B}\subseteq \mc{A}_j} f_{\Lambda_{j+1}}(\mc{B})$, $\mc{A}_{j+1}$ is a minimizer for $\min_{\mc{B}\subseteq \mc{F}} f_{\Lambda_{j+1}}(\mc{B})$ as well. This finishes the proof of the claim, and thus the proof of \eqref{eq:toprove}.

It remains to show the claimed runtime. As noted at the start of the proof, the while loop is executed at most $n$ times. The time one iteration takes is determined by line \ref{line:minimzer} of  Algorithm \ref{alg:Carleson}, the minimization of the submodular function $f_{\Lambda_j}$. Using the algorithm from \cite{JLSZ24}, a minimizer for a submodular function $f$ can be found with  runtime  $O(n^3 \log n \cdot \mathrm{EO}_{f}+ n^4 \log n)$, where $ \mathrm{EO}_{f}$ denotes the runtime of one function evaluation of $f$.
 If we initialize by storing the value of $\mu(Q)$ for all $Q \in \mc{F}$, a function evaluation of $f_{\Lambda_j}$ has a runtime of $O(\mathrm{EO}_1+n)$.
 Hence, the runtime of line \ref{line:minimzer} of  Algorithm~\ref{alg:Carleson} is $$O(n \cdot \mathrm{EO}_1+ n^3 \log n \cdot (\mathrm{EO}_1+n)+ n^4 \log n) = O(n^3 \log n \cdot \mathrm{EO}_1+ n^4 \log n).$$
This finishes the proof.
\end{proof}

\begin{remark}
  The runtime of Algorithm \ref{alg:Carleson} is, roughly speaking, $n$ times the best available algorithm for submodular function optimization on a set with $n$ elements. If one is  interested in the minimum number of  calls to the oracle \ref{O1}, one can replace the submodular function optimization  algorithm from \cite{JLSZ24} by one of the algorithms from \cite{Ji22}, yielding a runtime of Algorithm \ref{alg:Carleson} of either $O(n^4 \frac{\log \log n}{\log n} \cdot \mathrm{EO}_1+ n^9 \log n)$ or $O(n^3 \log n \cdot \mathrm{EO}_1+ 2^{O(n)})$.
\end{remark}

\section{The Carleson condition implies sparseness}\label{sec:sparse}
Having established an algorithm to compute the optimal $\Lambda$ such that $\mc{F}$ satisfies the $\Lambda$-Carleson condition, we now turn to our second goal: Constructing the  sets $\{E_Q\}_{Q\in\mathcal{F}}$ such that $\mu(E_Q) \geq \frac{1}{\Lambda} \mu(Q)$ for all $Q \in \mc{F}$. In particular, we  will provide a constructive proof that a finite collection satisfying the $\Lambda$-Carleson condition is $\frac{1}{\Lambda}$-sparse.

Without any assumptions on the measure space $(\Omega,\Sigma,\mu)$, not every $\Lambda$-Carleson condition is $\frac{1}{\Lambda}$-sparse. This is easily seen by taking $S=\cbrace{1,2}$, $\mu$ the counting measure and $\mc{F} = \mc{P}(S)$. This $\mc{F}$ satisfies the $2$-Carleson condition,  but is not $\frac12$-sparse. This problem cannot arise when $(\Omega,\Sigma,\mu)$ is non-atomic, which is assumed in, e.g., \cite{Haen18,Ve96}. Alternatively, to include the case of atomic measures, one can extend the notion of sparseness as in \cite[Definition 1]{Re22}. 

\begin{definition}\label{def:gensparse}
  We call $\mc{F}$ \emph{generalized $\eta$-sparse} for $\eta \in (0,1]$ if for every $Q \in \mc{F}$ there is a function $\varphi_Q\colon \Omega \to [0,1]$ such that
\begin{align*}
  \int_Q \varphi_Q\dd \mu &\geq \eta \mu(Q),\qquad Q \in \mc{F},\\
  \sum_{Q \in \mc{F}}\varphi_Q &\leq 1.
\end{align*}
\end{definition}
Any $\eta$-sparse collection $\mc{F}$ is generalized $\eta$-sparse by setting $\varphi_Q = \ind_{E_Q}$ for $Q \in \mc{F}$. Conversely, if $(\Omega,\Sigma,\mu)$ is non-atomic, a convexity argument as in \cite[Lemma 2.3]{Do75} shows that a generalized $\eta$-sparse collection is $\eta$-sparse.

\bigskip

Given a finite, nonempty collection $\mc{F}$ of sets in $\Sigma$ with $0<\mu(Q)<\infty$ for all $Q \in \mc{F}$ satisfying the $\Lambda$-Carleson condition, we will:
\begin{enumerate}[(i)]
  \item Provide an algorithm to construct the functions $\cbrace{\varphi_Q}_{Q \in \mc{F}}$  as in the definition of generalized $\Lambda^{-1}$-sparseness.
  \item Assuming that $(\Omega,\Sigma,\mu)$ is non-atomic, adapt our algorithm to construct the sets $\{E_Q\}_{Q\in\mathcal{F}}$ as in the definition of $\Lambda^{-1}$-sparseness.
\end{enumerate}
These algorithms will be presented in Subsections  \ref{subs:algorithms} and \ref{subs:algorithms2}. As preparation, we will first partition $\Omega$ according to the sets in $\mc{F}$ in Subsection \ref{subs:partition}, construct a weighted directed graph based on this partition in Subsection \ref{subs:graph} and provide some preliminaries on the maximum flow problem in Subsection \ref{subs:maxflow}.

\subsection{Partitioning \texorpdfstring{$\Omega$}{O}}
\label{subs:partition}
Without loss of generality, we may assume that $\Omega = \bigcup_{Q \in \mc{F}}Q$. For $\mc{A} \subseteq \mc{F}$, define
$$
E_{\mc{A}}:= \has{\bcap{Q \in \mc{A}}Q}\cap \has{\bcap{Q \in \mc{F}\setminus \mc{A}}Q^{\mathrm{c}}}.
$$
Since neither the $\Lambda$-Carleson condition nor $\eta$-sparseness is affected by  $\mu$-null sets, we may further assume that either $E_{\mc{A}}=\emptyset$ or $\mu(E_{\mc{A}})>0$ for all $\mc{A} \subseteq \mc{F}$. Define $$\mc{P}_0(\mc{F}):= \cbraceb{\mc{A} \in \mc{P}(\mc{F}):E_{\mc{A}} \neq \emptyset},$$ and note that the sets $E_{\mc{A}}$ for $\mc{A} \in \mc{P}_0(\mc{F})$ are the atoms of the $\sigma$-algebra generated by $\mc{F}$. Therefore, they form a partition in the sense that $E_{\mc{A}}$ and $E_{\mc{B}}$ are disjoint for $\mc{A} \neq \mc{B}$ and
\begin{align}\label{eq:partition}
  \bcup{Q \in \mc{A}} Q &= \bigcup_{\mc{B} \in \mc{P}_0({\mc{F}}):\mc{A}\cap \mc{B}\neq \emptyset} E_{\mc{B}}, && \mc{A} \in \mc{P}(\mc{F}).
\end{align}

\begin{example}\label{example:three-rectangles}
Consider $\Omega:= \R^2$ and let $\mu$ be the Lebesgue measure. Let $\mc{F}$ consist of the rectangles
\[
Q_1=[0,2]\times [0,\tfrac32],\qquad
Q_2=[1,3]\times [\tfrac12,2],\qquad
Q_3=[0,3]\times [0,\tfrac52],
\]
and write, for example, $E_{13}$ for the atom $E_{\cbrace{Q_1,Q_3}}$. The rectangles and the corresponding atoms are shown in Figure \ref{fig:rectangles}.
\begin{figure}[htbp]
  \centering
  \begin{tikzpicture}[scale=1.5, tick/.style={font=\scriptsize,black!100}]
    \draw[->,black!100] (-0.15,0) -- (3.35,0);
    \draw[->,black!100] (0,-0.15) -- (0,2.85);
    \foreach \x/\lab in {0/0,1/1,2/2,3/3}{\draw[black!100] (\x,0) -- (\x,-0.05) node[tick,below] {$\lab$};}
    \foreach \y/\lab in {0/0,1/{1},2/2}{\draw[black!100] (0,\y) -- (-0.05,\y) node[tick,left] {$\lab$};}
    \draw[fill=qthree!70, draw=black] (0,0) rectangle (3,2.5);
    \draw[fill=qone!70, draw=black] (0,0) rectangle (2,1.5);
    \draw[fill=qtwo!60, draw=black,  opacity=0.9] (1,0.5) rectangle (3,2);
    \draw[fill=none, draw=black] (1,0.5) rectangle (3,2);
    \node[anchor=south east] at (0.82,0.48) {$Q_1$};
    \node[anchor=north east] at (2.8,1.45) {$Q_2$};
    \node[anchor=south east] at (0.82,1.72) {$Q_3$};
  \end{tikzpicture}
  \hspace{0.8cm}
  \begin{tikzpicture}[scale=1.5, tick/.style={font=\scriptsize,black!100}]
    \draw[->,black!100] (-0.15,0) -- (3.35,0);
    \draw[->,black!100] (0,-0.15) -- (0,2.85);
    \foreach \x/\lab in {0/0,1/1,2/2,3/3}{\draw[black!100] (\x,0) -- (\x,-0.05) node[tick,below] {$\lab$};}
    \foreach \y/\lab in {0/0,1/{1},2/2}{\draw[black!100] (0,\y) -- (-0.05,\y) node[tick,left] {$\lab$};}
    \draw[fill=atomD!70, draw=black] (0,0) rectangle (3,2.5);
    \draw[fill=atomB!80, draw=black] (0,0) -- (2,0) -- (2,0.5) -- (1,0.5) -- (1,1.5) -- (0,1.5) -- cycle;
    \draw[fill=atomC!80, draw=black] (3,0.5) -- (3,2) -- (1,2) -- (1,1.5) -- (2,1.5) -- (2,0.5) -- cycle;
    \draw[fill=atomA!80, draw=black] (1,0.5) rectangle (2,1.5);
    \node at (1.5,1) {$E_{123}$};
    \node at (0.52,0.78) {$E_{13}$};
    \node at (2.5,1.25) {$E_{23}$};
    \node at (0.5,2.) {$E_3$};
    \node at (2.5,0.25) {$E_3$};
  \end{tikzpicture}
    \caption{The collection $\mathcal{F} = \cbrace{Q_1,Q_2,Q_3}$ and the sets $E_{\mc{A}}$ for $\mc{A} \in \mc{P}_0(\mc{F})$.}
 \label{fig:rectangles}
\end{figure}

The measures of the rectangles are $\mu\ha{Q_1}=\mu\ha{Q_2}=3$ and $\mu\ha{Q_3}=\frac{15}{2}$, while the measures of the atoms are
\[
\mu\ha{E_3}=\tfrac52,\qquad \mu\ha{E_{13}}=2,\qquad
\mu\ha{E_{23}}=2,\qquad \mu\ha{E_{123}}=1.
\]
Calculating all possibilities, we see that
\begin{equation*}
  \Lambda_{\mc{F}}= \frac{\mu(Q_1)+\mu(Q_2)+\mu(Q_3)}{\mu(Q_3)} = \frac{9}{5}<2.
\end{equation*}
In particular, $\mc{F}$ satisfies the $2$-Carleson condition, which we shall use below to showcase our algorithm.
\end{example}

%

The functions $\cbrace{\varphi_Q}_{Q \in \mc{F}}$ in the definition of generalized sparseness may need to take different values on each $E_\mc{A}$ for $\mc{A} \in \mc{P}_0(\mc{F})$. Similarly, the sets $\cbrace{E_Q}_{Q \in \mc{F}}$ in the definition of sparseness may require a specific amount of overlap with each of these sets. Thus, our algorithm will require the measure of  $E_\mc{A}$ for all $\mc{A} \in \mc{P}_0(\mc{F})$ as input. By the next lemma, these measures can be computed using the oracle \ref{O1}.

\begin{lemma}\label{lem:measureEA}
Let $\mc{F}$ be a finite collection of sets in $\Sigma$ for which $\mc{P}_0(\mc{F})$ is given.
 The measures $\mu(E_{\mc{A}})$ for all $\mc{A} \in \mc{P}_0(\mc{F})$ can be computed using \emph{\ref{O1}} with a runtime  $O(m\cdot \mathrm{EO}_1+ m^2)$, where  $m = \abs{\mc{P}_0(\mc{F})}$.
\end{lemma}

\begin{proof}
  For any $\mc{A} \in \mc{P}_0(\mc{F})$, we have by \eqref{eq:partition}
  $$
  \hab{\bcup{Q \in \mc{F}} Q}  \setminus \hab{\bcup{Q \in \mc{F} \setminus \mc{A}} Q} = \bigcup_{\mc{B} \in \mc{P}_0({\mc{F}}):\mc{B} \subseteq \mc{A}} E_{\mc{B}}.
  $$
  Since the $E_{\mc{B}}$ for $\mc{B} \in \mc{P}_0(\mc{F})$ are disjoint, this implies
  \begin{align*}
    \mu(E_{\mc{A}}) = \mu\hab{\bcup{Q \in \mc{F}}Q} - \mu\hab{\bcup{Q \in \mc{F}\setminus \mc{A}}Q} - \sum_{\mc{B} \in \mc{P}_0(\mc{F}):\mc{B} \subsetneq\mc{A}} \mu(E_{\mc{B}}).
  \end{align*}
  The first two terms on the right-hand side can be obtained from \ref{O1}.
 Therefore, we can compute $\mu(E_{\mc{A}})$ by induction on the number of sets in $\mc{A}$. To compute   $\mu(E_{\mc{A}})$ we need an oracle call and $O(m)$ additional work.
\end{proof}

 In general, $\mc{P}_0(\mc{F})$ can be equal to $\mc{P}(\mc{F})$ and thus contain $2^{\abs{\mc{F}}}$ sets. In the case of dyadic cubes or axis-parallel rectangles, which are important for applications in harmonic analysis, $\abs{\mc{P}_0(\mc{F})}$  is  polynomial in $\abs{\mc{F}}$. Hence, algorithms that have a runtime with polynomial dependence on $\abs{\mc{P}_0(\mc{F})}$ for such $\mc{F}$ have a runtime with polynomial dependence  on $\abs{\mc{F}}$.

\begin{example}\label{example:cubes}
Let $\ms{D}$ be the collection of dyadic cubes in $\R^d$, i.e.
\begin{align*}
  \ms{D} &:= \cbraceb{2^j\hab{[0,1)^d+k}:j \in \Z,k \in \Z^d}.
\end{align*}
For a finite collection $\mc{F} \subseteq \ms{D}$ we have $\abs{\mc{P}_0(\mc{F})}\leq \abs{\mc{F}}$.
Indeed, for $Q_1,Q_2\in \ms{D}$ we have either $Q_1 \subseteq Q_2$,
$Q_2 \subseteq Q_1$ or $Q_1\cap Q_2 = \emptyset$, which means that for $\mc{A} \subseteq \mc{F}$ we have $E_{\mc{A}} \neq \emptyset $ if and only if there is a $Q \in \mc{A}$ such that  $$\mc{A} = \cbrace{Q' \in \mc{F}:Q\subseteq Q'}.$$
Hence, we can index the elements of $\mc{P}_0(\mc{F})$ by elements of $\mc{F}$.
\end{example}

\begin{example}\label{example:rectangles}
  For a finite collection $\mc{F}$ of axis-parallel rectangles in $\R^d$ we have $\abs{\mc{P}_0(\mc{F})} \leq (2\abs{\mc{F}})^d$. Indeed, for $d=1$  the endpoints of the intervals partition $\bigcup_{I \in \mc{F}}I$ into at most $2\abs{\mc{F}}-1$ intervals. The general case follows by arguing in each coordinate separately.
\end{example}

\subsection{A weighted directed graph based on \texorpdfstring{$\mc{F}$}{F}}\label{subs:graph}
To construct the functions $\cbrace{\varphi_Q}_{Q \in \mc{F}}$ or the sets $\cbrace{E_Q}_{Q \in \mc{F}}$ as in the definition of (generalized) sparseness, we will build a weighted directed graph $G=(V,E,c)$ based on the collection $\mc{F}$ and the Carleson constant $\Lambda$. Here $V$ denotes the vertex set, $E$ the edge set and $c$ the capacity of the edges. Let us start by defining the vertex set $V$.
\begin{enumerate}[($\mathrm{V}_1$)]
  \item : We add a source vertex $s$ and a sink vertex $t$. \label{V1}
  \item : For every $Q \in \mc{F}$, we add a vertex $v_Q$. \label{V2}
  \item : For every $\mc{A} \in \mc{P}_0(\mc{F})$, we add a vertex $v_{\mc{A}}$. \label{V3}
\end{enumerate}
Note that for $Q \in \mc{F}$ we have both a vertex $v_Q$ and a vertex $v_{\cbrace{Q}}$ if $E_{\cbrace{Q}}\neq \emptyset$. Next, we define the edge set $E \subseteq V\times V$ and their capacity $c\colon E \to [0,\infty]$.
\begin{enumerate}[($\mathrm{E}_1$)]
  \item : For $\mc{A} \in \mc{P}_0(\mc{F})$, we add $(s, v_\mc{A})$ with  $c(s, v_\mc{A}) = \mu(E_{\mc{A}})$. \label{E1}
  \item : For  $\mc{A} \in \mc{P}_0(\mc{F})$ and $Q \in \mc{A}$, we add $(v_\mc{A}, v_Q)$ with $c(v_\mc{A}, v_Q) = \infty$. \label{E2}
  \item : For  $Q \in \mc{F}$, we add $(v_Q, t)$ with  $c(v_Q, t) = \Lambda^{-1} \mu(Q)$. \label{E3}
\end{enumerate}
We will run a maximum flow algorithm on the graph $G$ to construct the functions $\cbrace{\varphi_Q}_{Q \in \mc{F}}$ and the sets $\cbrace{E_Q}_{Q \in \mc{F}}$ as in the definition of (generalized) sparseness in Subsection \ref{subs:algorithms}. The idea is that the edges \ref{E1} provide $\mu(E_\mc{A})$ to each collection $\mc{A} \in \mc{P}_0(\mc{F})$, which they can distribute among the $Q \in \mc{A}$ via edges \ref{E2}. The edges \ref{E3} can be seen as a demand to provide $\Lambda^{-1}\mu(Q)$ to $Q \in \mc{F}$.

Returning to Example \ref{example:three-rectangles}, the graph $G$ based on $\mc{F}$ and $\Lambda=2$ is depicted in Figure \ref{fig:graph}. In the figure we abbreviate, for example, $v_{123}:= v_{\cbrace{Q_1,Q_2,Q_3}}$.

\begin{figure}[htbp]
  \centering
  \resizebox{0.9\textwidth}{!}{%
  \begin{tikzpicture}[
    >=Latex,
    vertex/.style={draw,circle,fill=black,minimum size=5pt,inner sep=0pt},
    atomlabel/.style={font=\scriptsize,fill=white,fill opacity=1,text opacity=1,inner sep=1pt},
    qlabel/.style={font=\scriptsize,fill=white,fill opacity=.86,text opacity=1,inner sep=1pt},
    caplabel/.style={font=\scriptsize,fill=white,fill opacity=.88,text opacity=1,inner xsep=1.4pt,inner ysep=.6pt},
    edge/.style={-{Latex[length=2mm]},draw=black!85}
  ]
    \node[vertex] (s) at (0,0) {};
    \node[vertex, fill=atomA] (a123) at (2.05,1.5) {};
    \node[vertex, fill=atomB] (a13) at (2.05,0.5) {};
    \node[vertex, fill=atomC] (a23) at (2.05,-0.5) {};
    \node[vertex, fill=atomD] (a3) at (2.05,-1.5) {};
    \node[vertex, fill=qone] (q1) at (5.45,1.05) {};
    \node[vertex, fill=qtwo] (q2) at (5.45,0) {};
    \node[vertex, fill=qthree] (q3) at (5.45,-1.05) {};
    \node[vertex] (t) at (7.45,0) {};

    \node[anchor=east] at (s.west) {$s$};
    \node[atomlabel,anchor=east] at (1.95,1.68) {$v_{123}$};
    \node[atomlabel,anchor=east] at (1.95,0.68) {$v_{13}$};
    \node[atomlabel,anchor=east] at (1.95,-0.68) {$v_{23}$};
    \node[atomlabel,anchor=east] at (1.95,-1.68) {$v_3$};
    \node[qlabel,anchor=west] at (5.55,1.23) {$v_{Q_1}$};
    \node[qlabel,anchor=west] at (5.55,0.18) {$v_{Q_2}$};
    \node[qlabel,anchor=west] at (5.55,-1.23) {$v_{Q_3}$};
    \node[anchor=west] at (t.east) {$t$};

    \draw[edge] (s) -- node[caplabel,pos=.5,above, yshift=3pt] {$1$} (a123);
    \draw[edge] (s) -- node[caplabel,pos=.5,above, yshift=1pt] {$2$} (a13);
    \draw[edge] (s) -- node[caplabel,pos=.5,below, yshift=-1pt] {$2$} (a23);
    \draw[edge] (s) -- node[caplabel,pos=.5,below, yshift=-3pt] {$\frac52$} (a3);

    \draw[edge] (a123) -- (q1);
    \draw[edge] (a123) -- (q2);
    \draw[edge] (a123) -- (q3);
    \draw[edge] (a13) -- (q1);
    \draw[edge] (a13) -- (q3);
    \draw[edge] (a23) -- (q2);
    \draw[edge] (a23) -- (q3);
    \draw[edge] (a3) -- (q3);

    \draw[edge] (q1) -- node[caplabel,pos=.5,above, yshift=3pt] {$\frac32$} (t);
    \draw[edge] (q2) -- node[caplabel,pos=.5,below, yshift=-1pt] {$\frac32$} (t);
    \draw[edge] (q3) -- node[caplabel,pos=.5,below, yshift=-3.5pt] {$\frac{15}{4}$} (t);
  \end{tikzpicture}}
  \caption{The weighted graph $G$ based on the collection $\mc{F}$ from Example \ref{example:three-rectangles} and $\Lambda=2$. Infinite capacities are not depicted.}
  \label{fig:graph}
\end{figure}

\subsection{The maximum flow problem}\label{subs:maxflow}
Before we continue, we provide some background on the maximum flow problem. Let $G=(V,E,c)$ be a weighted directed graph with a source $s \in V$ and a sink $t \in V$.  A function $f \colon E \to [0,\infty)$ is called an \emph{flow} through the graph  $G$ if it satisfies  \emph{flow conservation}
\begin{align}\label{eq:flowsum}
  \sum_{u\in V: (u,v) \in E} f(u,v) &= \sum_{w\in V: (v,w) \in E} f(v,w) && v \in V\setminus \cbrace{s,t} \intertext{and the \emph{capacity constraint} imposed by $c$}
\label{eq:flowcap}
  f(u,v) &\leq c(u,v), &&(u,v) \in E.
\end{align}
The \emph{value of a flow $f$} is the amount of flow  from $s$ to $t$, which is given by
\begin{equation}\label{eq:flowdef}
  \abs{f}:= \sum_{v \in V: (s,v) \in E} f(s,v) = \sum_{v \in V: (v,t) \in E} f(v,t)
\end{equation}
where the equality follows from flow conservation.

Finding a maximum flow of a graph is a classical optimization problem, for which (strongly) polynomial algorithms date back to Ford--Fulkerson \cite{FF56}, Edmonds--Karp \cite{EK72} and Dinic \cite{Di70}. The best runtime for our purposes is provided by Orlin's algorithm \cite{Or13}, which can find a maximum flow through a graph with runtime $O(\abs{V}\abs{E})$.

 To make a connection with the Carleson condition, we also need to introduce cuts. An \emph{$(s,t)$-cut} is a partition $(S,T)$ of $V$ (so $T = V \setminus S$) such that $s \in S$ and $t \in T$. The \emph{capacity} of the $(s,t)$-cut is given by
$$
c(S,T) = \sum_{(u,v) \in E: u \in S, v \in T} c(u,v).
$$
We will see that a minimum capacity $(s,t)$-cut of the graph constructed in Subsection \ref{subs:graph} corresponds to  a
maximizer in the Carleson condition.

\bigskip

Intuitively, the value of a flow through a graph can never exceed the smallest bottleneck in the graph, i.e.,  the minimum capacity of a $(s,t)$-cut. In fact, the maximum value of a  flow is  equal to the minimum capacity of a  $(s,t)$-cut, which is the famous max-flow min-cut theorem.

\begin{prop}\label{prop:maxflowmincut}
  Let $G=(V,E,c)$ be a weighted directed graph with source $s\in V$ and sink $t \in V$. Then
  \begin{align*}
    \max\cbraceb{\abs{f}: f \text{ is a flow}} = \min\cbraceb{c(S,T):(S,T) \text{ is an $(s,t)$-cut}}.
  \end{align*}
\end{prop}

\subsection{Constructing the \texorpdfstring{$\varphi_Q$}{phiQ}} \label{subs:algorithms}
We are now ready to provide the algorithm to construct the
functions $\cbrace{\varphi_Q}_{Q \in \mc{F}}$  as in the definition of generalized sparseness.

\begin{algorithm}\label{algorithmFindphiQ}
\DontPrintSemicolon
    \caption{Construct $\varphi_Q$.}
        \SetKwInOut{Input}{Input}\SetKwInOut{Output}{Output}
    \Input{A finite collection $\mc{F}$   satisfying the $\Lambda$-Carleson condition.}
    \Output{The functions $\varphi_Q$ for $Q\in \mc{F}$.}
    { Construct the  graph $G=(V,E,c)$ as in Subsection \ref{subs:graph}.\label{line:constructG}\;
    Compute a maximum flow $f \colon E \to [0,\infty)$ from $s$ to $t$.\label{line:maxflow}\;
    }
    \For{$Q \in \mc{F}$
    }{
    Set $\varphi_Q = \displaystyle\sum_{\mc{A} \in \mc{P}_0(\mc{F}): Q \in \mc{A}} \frac{f(v_{\mc{A}},v_Q)}{\mu(E_{\mc{A}})}\ind_{E_{\mc{A}}}$. \label{line:defphi}
    }
\end{algorithm}

We will show that the output functions $\cbrace{\varphi_Q}_{Q \in \mc{F}}$  from Algorithm \ref{algorithmFindphiQ} satisfy Definition \ref{def:gensparse} with $\eta = \frac{1}{\Lambda}$, which proves that $\mc{F}$ is generalized $\frac{1}{\Lambda}$-sparse.
\begin{theorem}\label{thm:phiQ}
 Let $\mc{F}$ be a finite collection of sets in $\Sigma$ with $0<\mu(Q)<\infty$ for all $Q \in \mc{F}$ and let $\mc{P}_0(\mc{F})$ be given. Suppose that $\mc{F}$ satisfies the $\Lambda$-Carleson condition. The output $\cbrace{\varphi_Q}_{Q \in \mc{F}}$ of Algorithm \ref{algorithmFindphiQ} satisfies
 \begin{align}\label{eq:toprovealg21}
  \int_Q \varphi_Q\dd \mu &\geq \tfrac1\Lambda\mu(Q),\qquad Q \in \mc{F}\\
  \sum_{Q \in \mc{F}}\varphi_Q &\leq 1.\label{eq:toprovealg22}
\end{align}
The runtime of Algorithm  \ref{algorithmFindphiQ}  is $O(m\cdot \mathrm{EO}_1+ n^2m+nm^2 )$, where $n=\abs{\mc{F}}$ and $m = \abs{\mc{P}_0(\mc{F})}$.
\end{theorem}

\begin{proof}
Let $\cbrace{\varphi_Q}_{Q \in \mc{F}}$ be the output of Algorithm \ref{algorithmFindphiQ}.
For $x \in \Omega$ we have by Line \ref{line:defphi} of Algorithm \ref{algorithmFindphiQ}, the flow conservation in \eqref{eq:flowsum}, the capacity constraint \eqref{eq:flowcap} for the edges \ref{E1} and the disjointness of the $E_{\mc{A}}$'s, that
\begin{align*}
  \sum_{Q \in \mc{F}}\varphi_Q(x) &= \sum_{Q \in \mc{F}}\sum_{\mc{A} \in \mc{P}_0(\mc{F}): Q \in \mc{A}} \frac{f(v_{\mc{A}},v_Q)}{\mu(E_{\mc{A}})}\ind_{E_{\mc{A}}}(x)\\
  &= \sum_{\mc{A} \in \mc{P}_0(\mc{F})}\sum_{Q \in \mc{A}} \frac{f(v_{\mc{A}},v_Q)}{\mu(E_{\mc{A}})}\ind_{E_{\mc{A}}}(x)\\
  &= \sum_{\mc{A} \in \mc{P}_0(\mc{F})} \frac{f(s,v_{\mc{A}})}{\mu(E_{\mc{A}})}\ind_{E_{\mc{A}}}(x)\\
  &\leq \sum_{\mc{A} \in \mc{P}_0(\mc{F})} \ind_{E_{\mc{A}}}(x) \leq 1,
\end{align*}
proving \eqref{eq:toprovealg22}.

For \eqref{eq:toprovealg21}, let  $Q \in \mc{F}$. By the flow conservation in \eqref{eq:flowsum}, we have
\begin{align*}
  \int_Q \varphi_Q\dd \mu = \sum_{\mc{A} \in \mc{P}_0(\mc{F}): Q \in \mc{A}} f(v_{\mc{A}},v_Q) = f(v_Q,t).
\end{align*}
Hence, we need to show that
$f(v_Q,t) \geq \Lambda^{-1}\mu(Q) = c(v_Q,t)$. Since $f \leq c$,  it suffices to show that
$$
\sum_{Q \in \mc{F}} f(v_Q,t) \geq \sum_{Q\in \mc{F}} \Lambda^{-1}\mu(Q).
$$
As the left-hand side is equal to $\abs{f}$ and $f$ is a maximum flow, by Proposition \ref{prop:maxflowmincut} it is equivalent to show that any $(s,t)$-cut  has capacity at least $\sum_{Q\in \mc{F}} \Lambda^{-1}\mu(Q)$.

Let $(S,T)$ be an $(s,t)$-cut with finite capacity. Define \begin{align*}
  \mc{A}_S&:=\cbrace{Q\in \mc{F}:v_Q \in S},\\
  \mc{A}_T&:=\cbrace{Q\in \mc{F}:v_Q \in T},
\end{align*} and note that for $\mc{B} \in \mc{P}_0(\mc{F})$ with $\mc{A}_T\cap \mc{B} \neq \emptyset$ we must have $v_{\mc{B}} \in T$, since the edges \ref{E2} have infinite capacity. Using this observation, \eqref{eq:partition} and the $\Lambda$-Carleson condition, we have
\begin{align*}
c(S,T) &\geq \sum_{Q \in \mc{A}_S} c(v_Q,t) + \sum_{\mc{B} \in \mc{P}_0(\mc{F}):\mc{A}_T\cap \mc{B} \neq \emptyset}  c(s,v_{\mc{B}})\\
&= \sum_{Q \in \mc{A}_S}\Lambda^{-1}\mu(Q) + \sum_{\mc{B} \in \mc{P}_0(\mc{F}):\mc{A}_T\cap \mc{B} \neq \emptyset}  \mu(E_{\mc{B}})\\
&= \sum_{Q \in \mc{A}_S} \Lambda^{-1}\mu(Q) + \mu \has{\bcup{Q \in \mc{A}_T} Q}\geq \sum_{Q \in \mc{A}_S\cup \mc{A}_T} \Lambda^{-1}\mu(Q).
\end{align*}
This shows that any $(s,t)$-cut has capacity at least $\sum_{Q\in \mc{F}} \Lambda^{-1}\mu(Q)$, finishing the proof of
\eqref{eq:toprovealg21}.

It remains to show the claimed runtime. By Lemma \ref{lem:measureEA} we know that line \ref{line:constructG} of Algorithm \ref{algorithmFindphiQ} has a runtime of  $O(m\cdot \mathrm{EO}_1+ m^2)$. Furthermore, we have
\begin{align*}
  \abs{V} &= \abs{\mc{P}_0(\mc{F})} +\abs{\mc{F}}+2 = O(m+n),\\
  \abs{E} &\leq \abs{\mc{P}_0(\mc{F})} + \abs{\mc{F}}\cdot\abs{\mc{P}_0(\mc{F})} + \abs{\mc{F}} = O(nm).
\end{align*}
Using Orlin's algorithm \cite{Or13}, line \ref{line:maxflow} of Algorithm \ref{algorithmFindphiQ} has a runtime of $O(\abs{V}\abs{E}) = O(n^2m+nm^2)$. Finally, line \ref{line:defphi} of Algorithm \ref{algorithmFindphiQ} is executed $n$ times and has a runtime of $O(m)$. Combined, this yields the claimed runtime.
\end{proof}

\begin{remark}
  If $\mc{F}$ consists of (dyadic) cubes or rectangles in $\R^d$, Algorithm~\ref{algorithmFindphiQ} is strongly polynomial by Examples \ref{example:cubes} and \ref{example:rectangles}.
\end{remark}

\begin{remark}\label{rem:infinite}
The key ingredient in the proof of Theorem \ref{thm:phiQ} is the max-flow min-cut theorem. This theorem has been generalized to infinite graphs  \cite{ABGPS11,Loc22}, albeit only for graphs with a countable number of edges. As a consequence, for a \emph{countable} collection of sets $\mc{F}$ in $\Sigma$ such that $\mc{P}_0(\mc{F})$ is countable,  we can  show the equivalence between the $\Lambda$-Carleson condition and generalized $\frac1\Lambda$-sparseness as well. This provides an alternative proof to \cite{Haen18} for, e.g., $\mc{F}$ a countable collection of rectangles in $\R^d$. For the details, we refer to \cite[Chapter 4]{Hon24}.
\end{remark}

\begin{example}\label{example:three-rectangle-algorithms}
We return to the collection of rectangles $\mc{F}$ from Example \ref{example:three-rectangles}, for which we computed
$$
\Lambda_{\mc{F}}= \tfrac95 <2.
$$
We constructed the graph $G$ corresponding to $\mc{F}$ and $\Lambda=2$ in Figure \ref{fig:graph}. An example of a flow $f\colon E\to [0,\infty)$ through the graph $G$ satisfying the capacity constraints is provided in Table \ref{tab:1} and is depicted in Figure \ref{fig:flow-example}.

\begin{table}[htbp]
\centering
\caption{The flow values of $f$.}
\label{tab:1}
\renewcommand{\arraystretch}{1.25}

\begin{tabular}[t]{c}
{\scriptsize $s \to \mathcal A$}\\[2pt]
\begin{tabular}{c|cccc}
 & $v_{123}$ & $v_{13}$ & $v_{23}$ & $v_3$ \\ \hline
$s$ & $1$ & $2$ & $2$ & $\tfrac74$
\end{tabular}
\end{tabular}
\hspace{0.2cm}
\begin{tabular}[t]{c}
{\scriptsize $\mathcal A \to Q$}\\[2pt]
\begin{tabular}{c|ccc}
 & $v_{Q_1}$ & $v_{Q_2}$ & $v_{Q_3}$ \\ \hline
$v_{123}$ & $0$ & $0$ & $1$ \\
$v_{13}$  & $\tfrac32$ & $0$ & $\tfrac12$ \\
$v_{23}$  & $0$ & $\tfrac32$ & $\tfrac12$ \\
$v_3$     & $0$ & $0$ & $\tfrac74$
\end{tabular}
\end{tabular}
\hspace{0.2cm}
\begin{tabular}[t]{c}
{\scriptsize $Q \to t$}\\[2pt]
\begin{tabular}{c|c}
 & $t$ \\ \hline
$v_{Q_1}$ & $\tfrac32$ \\
$v_{Q_2}$ & $\tfrac32$ \\
$v_{Q_3}$ & $\tfrac{15}{4}$
\end{tabular}
\end{tabular}
\end{table}

\begin{figure}[htbp]
  \centering
  \resizebox{0.9\textwidth}{!}{%
  \begin{tikzpicture}[
    >=Latex,
    vertex/.style={draw,circle,fill=black,minimum size=5pt,inner sep=0pt},
    atomlabel/.style={font=\scriptsize,fill=white,fill opacity=1,text opacity=1,inner sep=1pt},
    qlabel/.style={font=\scriptsize,fill=white,fill opacity=.9,text opacity=1,inner sep=1pt},
    flowlabel/.style={font=\scriptsize,fill=white,fill opacity=.95,text opacity=1,inner xsep=1.4pt,inner ysep=.6pt},
    zeroflow/.style={-{Latex[length=2mm]},draw=black!85},
    posflow/.style={-{Latex[length=2mm]},draw=flowcolor!50,line width=1pt},
    satflow/.style={-{Latex[length=2mm]},draw=flowcolor,line width=1pt}
  ]

    \node[vertex] (s) at (0,0) {};
    \node[vertex, fill=atomA] (a123) at (2.05,1.5) {};
    \node[vertex, fill=atomB] (a13)  at (2.05,0.5) {};
    \node[vertex, fill=atomC] (a23)  at (2.05,-0.5) {};
    \node[vertex, fill=atomD] (a3)   at (2.05,-1.5) {};
    \node[vertex, fill=qone]   (q1)  at (5.45,1.05) {};
    \node[vertex, fill=qtwo]   (q2)  at (5.45,0) {};
    \node[vertex, fill=qthree] (q3)  at (5.45,-1.05) {};
    \node[vertex] (t) at (7.45,0) {};

    \node[anchor=east] at (s.west) {$s$};
    \node[atomlabel,anchor=east] at (1.95,1.68) {$v_{123}$};
    \node[atomlabel,anchor=east] at (1.95,0.68) {$v_{13}$};
    \node[atomlabel,anchor=east] at (1.95,-0.68) {$v_{23}$};
    \node[atomlabel,anchor=east] at (1.95,-1.68) {$v_3$};
    \node[qlabel,anchor=west] at (5.55,1.23) {$v_{Q_1}$};
    \node[qlabel,anchor=west] at (5.55,0.18) {$v_{Q_2}$};
    \node[qlabel,anchor=west] at (5.55,-1.23) {$v_{Q_3}$};
    \node[anchor=west] at (t.east) {$t$};

    \draw[satflow] (s) -- node[flowlabel,pos=.5,above,yshift=3pt] {$1$} (a123);
    \draw[satflow] (s) -- node[flowlabel,pos=.5,above,yshift=1pt] {$2$} (a13);
    \draw[satflow] (s) -- node[flowlabel,pos=.5,below,yshift=-1pt] {$2$} (a23);
    \draw[posflow] (s) -- node[flowlabel,pos=.5,below,yshift=-3pt] {$\tfrac74$} (a3);

    \draw[zeroflow] (a123) -- (q1);
    \draw[zeroflow] (a123) -- (q2);
    \draw[posflow]  (a123) -- (q3);
    \draw[posflow]  (a13)  -- (q1);
    \draw[posflow]  (a13)  -- (q3);
    \draw[posflow]  (a23)  -- (q2);
    \draw[posflow]  (a23)  -- (q3);
    \draw[posflow]  (a3)   -- (q3);

    \draw[satflow] (q1) -- node[flowlabel,pos=.5,above,yshift=3pt] {$\tfrac32$} (t);
    \draw[satflow] (q2) -- node[flowlabel,pos=.5,below,yshift=-1pt] {$\tfrac32$} (t);
    \draw[satflow] (q3) -- node[flowlabel,pos=.5,below,yshift=-3.5pt] {$\tfrac{15}{4}$} (t);

  \end{tikzpicture}}
  \caption{The  flow $f$ through the graph $G$ in Figure \ref{fig:graph}. The flow values on the atom-to-set edges ($\mathrm{E}_2$) are omitted, see Table \ref{tab:1}. Saturated edges are drawn in dark blue and positive, non-saturated edges are drawn in light blue.}
  \label{fig:flow-example}
\end{figure}

The value of the flow $f$ is
\[
1+2+2+\tfrac74=\tfrac32+\tfrac32+\tfrac{15}{4}=\tfrac{27}{4},
\]
which is maximum. Indeed, the cut $S=V\setminus\cbrace{t}$ and $T=\cbrace{t}$ has capacity
\[
c(S,T)=\tfrac32+\tfrac32+\tfrac{15}{4}=\tfrac{27}{4},
\]
so Proposition \ref{prop:maxflowmincut} shows that no larger flow is possible.
Recalling that 
\[
\mu\ha{E_3}=\tfrac52,\qquad \mu\ha{E_{13}}=2,\qquad
\mu\ha{E_{23}}=2,\qquad \mu\ha{E_{123}}=1,
\]
Algorithm \ref{algorithmFindphiQ} returns
\[
\varphi_{Q_1}=\tfrac34\ind_{E_{13}},\qquad
\varphi_{Q_2}=\tfrac34\ind_{E_{23}},
\]
and
\[
\varphi_{Q_3}
=\ind_{E_{123}}
+\tfrac14\ind_{E_{13}}
+\tfrac14\ind_{E_{23}}
+ \tfrac{7}{10}\ind_{E_3}.
\]
These functions satisfy $\sum_{Q\in\mc{F}}\varphi_Q\leq 1$ and
\[
\int_{Q_j}\varphi_{Q_j}\dd\mu=\tfrac12\mu(Q_j),\qquad j=1,2,3.
\]
\end{example}

\subsection{Constructing the \texorpdfstring{$E_Q$}{EQ}} \label{subs:algorithms2}
It remains to provide a constructive proof of sets $\cbrace{E_Q}_{Q \in \mc{F}}$ under the additional assumption that  the measure space $(\Omega,\Sigma,\mu)$ is non-atomic. Using  a convexity argument as in \cite[Lemma 2.3]{Do75}, one can show the existence of sets $\{E_Q\}_{Q\in\mathcal{F}}$, given the functions $\cbrace{\varphi_Q}_{Q \in \mc{F}}$. However, this argument is based on the Krein-Milman theorem and is therefore non-constructive. Instead, we will adapt Algorithm \ref{algorithmFindphiQ} to construct the sets $\cbrace{E_Q}_{Q \in \mc{F}}$ explicitly. To do so, we need an additional oracle handling the assumption that our measure space is non-atomic. Concretely, we assume to have the following oracle available:
\begin{enumerate}[($\mathrm{O}_2$)]
  \item :  Given  $E\in \Sigma$ and $a \in [0,1]$, return $E_a \subseteq E$ with $\mu(E_a) = a \mu(E)$. \label{O2}
\end{enumerate}
The time needed for one call to the oracle \ref{O2} is denoted by $\mathrm{EO}_2$.

\begin{algorithm}[t]\label{algorithmFindEQ}
\DontPrintSemicolon
    \caption{Construct $E_Q$.}

    \SetKwInOut{Input}{Input}\SetKwInOut{Output}{Output}
    \Input{A finite collection $\mc{F}$   satisfying the $\Lambda$-Carleson condition.}
    \Output{The sets $E_Q$ for $Q\in \mc{F}$.}
    { Construct the  graph $G=(V,E,c)$ as in Subsection \ref{subs:graph}.\label{line:constructGEQ}\;
    Compute the maximum flow $f \colon E \to [0,\infty)$ from $s$ to $t$.\label{line:maxflowEQ}\;
    }
    \For{$Q \in \mc{F}$}{
     Set $E_Q = \emptyset$.
    }
    \For{$\mc{A} \in \mc{P}_0(\mc{F})$
    }{
    Set $F=E_{\mc{A}}$.\;
    \For{$Q \in \mc{A}$}{
     Take a set $F_{Q}\subseteq F$ with $\mu(F_{Q}) = {f(v_{\mc{A}},v_Q)}$.\label{line:takechunk}\;
    Update $E_Q \leftarrow  E_Q \cup  {F_{Q}}$.\label{line:updateEQ} \;
    Update $F\leftarrow F\setminus F_{Q}$.
    }   }
\end{algorithm}

We will  show that the output sets $\cbrace{E_Q}_{Q \in \mc{F}}$  from Algorithm \ref{algorithmFindEQ} satisfy the definition of $\eta$-sparseness with $\eta = \frac{1}{\Lambda}$, which proves that $\mc{F}$ is $\frac{1}{\Lambda}$-sparse.

\begin{theorem}\label{thm:EQ}
 Let $\mc{F}$ be a finite collection of sets in $\Sigma$ with $0<\mu(Q)<\infty$ for all $Q \in \mc{F}$ and let $\mc{P}_0(\mc{F})$ be given. Suppose that $\mc{F}$ satisfies the $\Lambda$-Carleson condition and $(\Omega,\Sigma,\mu)$ is non-atomic. The output $\cbrace{E_Q}_{Q \in \mc{F}}$ of Algorithm \ref{algorithmFindEQ} satisfies $E_Q \subseteq Q$ for all $Q \in \mc{F}$,
 \begin{align}\label{eq:toprovealg3}
  \mu(E_Q) &\geq \tfrac{1}{\Lambda} \mu(Q),\qquad Q \in \mc{F},
\end{align}
and the sets $\cbrace{E_Q}_{Q \in \mc{F}}$ are disjoint.
The runtime of Algorithm  \ref{algorithmFindEQ}  is $$O(m\cdot \mathrm{EO}_1+nm\cdot \mathrm{EO}_2+ n^2m+nm^2 ),$$ where $n=\abs{\mc{F}}$ and $m = \abs{\mc{P}_0(\mc{F})}$.
\end{theorem}

\begin{proof}
  We start by showing that line \ref{line:takechunk} of Algorithm \ref{algorithmFindEQ} is possible in every iteration of the for-loop. For this, we need to show that we always have $\mu(F) \geq f(v_{\mc{A}},v_Q)$. This follows directly from the fact that, by the flow conservation in \eqref{eq:flowsum}, we have
  $$
  \sum_{Q \in \mc{A}} f(v_{\mc{A}},v_Q) = f(s,v_{\mc{A}}) \leq c(s,v_{\mc{A}}) = \mu(E_{\mc{A}}).
  $$

  It is clear that  $E_Q \subseteq Q$ for all $Q \in \mc{F}$ and the sets $\cbrace{E_Q}_{Q \in \mc{F}}$ are disjoint, so it remains to prove \eqref{eq:toprovealg3}. By line \ref{line:updateEQ} of Algorithm \ref{algorithmFindEQ} and \eqref{eq:flowsum} we have
  $$
  \mu(E_Q) = \sum_{\mc{A} \in \mc{P}_0(\mc{F}):Q \in \mc{A}} f(v_{\mc{A}},v_Q) = f(v_Q,t).
  $$
  Exactly as in the proof of Theorem \ref{thm:phiQ}, one can now show that $f(v_Q,t) \geq \Lambda^{-1}\mu(Q)$, yielding \eqref{eq:toprovealg3}.

  Finally, the runtime analyses of lines  \ref{line:constructGEQ} and \ref{line:maxflowEQ} of Algorithm \ref{algorithmFindEQ} are the same as the runtime analysis of Algorithm \ref{algorithmFindphiQ} in the proof of Theorem \ref{thm:phiQ}. The rest of Algorithm \ref{algorithmFindEQ} consists of two nested for loops over $\mc{P}_0$ and $\mc{A}\in \mc{P}_0$ respectively. Each iteration has one oracle call to \ref{O2}, yielding a runtime of $O(nm\cdot \mathrm{EO}_2)$.
\end{proof}

\begin{example}\label{example:three-rectangle-EQ}
We continue the study of the family of rectangles $\mc{F}$ from Examples \ref{example:three-rectangles} and \ref{example:three-rectangle-algorithms}. Since Lebesgue measure is non-atomic, Algorithm \ref{algorithmFindEQ} constructs disjoint subsets $E_{Q_j}$ by partitioning each atom according to the flow $f$ described in Table \ref{tab:1}. One possible output is
\begin{align*}
E_{Q_1} &= (0,1)\times(0,\tfrac32), \\
E_{Q_2} &= (2,3)\times(\tfrac12,2), \\
E_{Q_3}
&= \bigl((1,2)\times(0,2)\bigr) 
\cup \bigl((0,3)\times(2,\tfrac52)\bigr) 
\cup \bigl((2,\tfrac52)\times(0,\tfrac12)\bigr).
\end{align*}
These sets are pairwise disjoint and
\begin{align*}
\mu(E_{Q_1})&=\tfrac32\hspace{4pt} = \tfrac12\,\mu(Q_1),\\
\mu(E_{Q_2})&=\tfrac32\hspace{4pt}= \tfrac12\,\mu(Q_{2}),\\
\mu(E_{Q_3})&=\tfrac{15}{4}= \tfrac12\,\mu(Q_3),
\end{align*}
i.e. they satisfy the conditions for $\frac12$-sparseness. 
This configuration is shown in the left panel of Figure \ref{fig:EQ-example}. Note that a part of $Q_1\cup Q_2\cup Q_3$ of positive measure remains unassigned. 

Since the maximizing subcollection $\mc{A}\subseteq \mc{F}$ for the Carleson constant of $\mc{F}$ is the full family $\mc{A} = \mc F$ (see Example \ref{example:three-rectangles}), one expects that for the optimal Carleson constant no part of positive measure will be left unassigned. This is indeed the case: running Algorithm \ref{algorithmFindEQ} with the optimal Carleson constant $\Lambda=\Lambda_{\mc F}=\tfrac95$  produces the output in the right panel of Figure \ref{fig:EQ-example}.

\begin{figure}[htbp]
  \centering
  \begin{tikzpicture}[scale=1.5, tick/.style={font=\scriptsize,black!100}]
    \draw[->,black!100] (-0.15,0) -- (3.35,0);
    \draw[->,black!100] (0,-0.15) -- (0,2.85);
    \foreach \x/\lab in {0/0,1/1,2/2,3/3}{\draw[black!100] (\x,0) -- (\x,-0.05) node[tick,below] {$\lab$};}
    \foreach \y/\lab in {0/0,1/{1},2/2}{\draw[black!100] (0,\y) -- (-0.05,\y) node[tick,left] {$\lab$};}
    \draw[fill=unusedcolor!55, draw=black] (2.5,0) rectangle (3,0.5);
    \draw[fill=unusedcolor!55, draw=black] (0,1.5) rectangle (1,2);
    \draw[fill=qone!70, draw=black] (0,0) rectangle (1,1.5);
    \draw[fill=qtwo!60, draw=black] (2,0.5) rectangle (3,2);
    \draw[fill=qthree!70, draw=black] (1,0) rectangle (2,2);
    \draw[fill=qthree!70, draw=black] (0,2) rectangle (3,2.5);
    \draw[fill=qthree!70, draw=black] (2,0) rectangle (2.5,0.5);
    \node at (0.5,0.75) {$E_{Q_1}$};
    \node at (2.5,1.25) {$E_{Q_2}$};
    \node at (1.5,1) {$E_{Q_3}$};
  \end{tikzpicture}
   \hspace{0.8cm}
  \begin{tikzpicture}[scale=1.5, tick/.style={font=\scriptsize,black!100}]
    \draw[->,black!100] (-0.15,0) -- (3.35,0);
    \draw[->,black!100] (0,-0.15) -- (0,2.85);
    \foreach \x/\lab in {0/0,1/1,2/2,3/3}{\draw[black!100] (\x,0) -- (\x,-0.05) node[tick,below] {$\lab$};}
    \foreach \y/\lab in {0/0,1/{1},2/2}{\draw[black!100] (0,\y) -- (-0.05,\y) node[tick,left] {$\lab$};}
    \draw[fill=qthree!70, draw=black] (0,0) rectangle (3,2.5);
    \draw[fill=qone!70, draw=black] (0,0) rectangle (1,1.5);
    \draw[fill=qone!70, draw=black] (1,0) rectangle (1.333,0.5);
    \draw[fill=qtwo!60, draw=black] (2,0.5) rectangle (3,2);
    \draw[fill=qtwo!60, draw=black] (1.667,1.5) rectangle (2,2);
    \draw[draw=black] (1.333,0) rectangle (2,0.5);
    \draw[draw=black] (1,1.5) rectangle (1.667,2);
    \draw[draw=black] (1,0.5) rectangle (2,1.5);
    \node at (0.5,0.75) {$E_{Q_1}$};
    \node at (2.5,1.25) {$E_{Q_2}$};
    \node at (1.5,1) {$E_{Q_3}$};
    \node at (1.5,2.25) {$E_{Q_3}$};
  \end{tikzpicture}
  \caption{Output of Algorithm \ref{algorithmFindEQ} for the collection of rectangles in Example \ref{example:three-rectangles}. The left panel corresponds to $\Lambda=2$, and the right panel to the optimal value $\Lambda=\Lambda_{\mc{F}}=\frac95$.}
  \label{fig:EQ-example}
\end{figure}
\end{example}

\subsection*{Declaration of generative AI and AI-assisted technologies in the manuscript preparation process}
During the preparation of this work the authors used ChatGPT 5.5 Extended Thinking in order to make the draft version of Figures \ref{fig:rectangles}-\ref{fig:EQ-example}. After using this tool, the authors reviewed and edited the content as needed and take full responsibility for the content of the published article.

\bibliographystyle{alpha}
\bibliography{bibliography}

\end{document}